\documentclass[11pt]{amsart} 
\usepackage[T1]{fontenc}
\usepackage[utf8]{inputenc}

\usepackage{amsfonts,amsmath,amsthm,amssymb}
\usepackage{enumerate} 
\usepackage{enumitem} 
\usepackage{fullpage}
\usepackage[hidelinks,linktoc=all]{hyperref}
\usepackage{pgf} 
\usepackage{graphicx}
\usepackage{adjustbox}
\usepackage{tikz}
\usetikzlibrary{positioning}

\usepackage{floatrow}
\DeclareMathOperator{\wilfoper}{W} %Wilf number
\DeclareMathOperator{\eliahouoper}{E} %Eliahou number
\DeclareMathOperator{\Frobeniusoper}{F} %Frobenius
\DeclareMathOperator{\multiplicityoper}{m} %multiplicity
\DeclareMathOperator{\conductoroper}{\chi} %conductor -- notation used by Wilf 
\DeclareMathOperator{\genusoper}{\Omega} %genus -- notation used by Wilf 
\DeclareMathOperator{\qoper}{q} %q
\DeclareMathOperator{\leftsoper}{L} %lefts
\DeclareMathOperator{\primitivesoper}{P} %primitives
\DeclareMathOperator{\embeddingdimensionoper}{d} %number of primitives %embeding dimension

\DeclareMathOperator{\ratiooper}{r} %ratio
\DeclareMathOperator{\typeoper}{t} %type
\DeclareMathOperator{\Aperyoper}{Ap} %Apery

\newtheorem{theorem}{Theorem}[section]

\newtheorem{corollary}[theorem]{Corollary}
\newtheorem{proposition}[theorem]{Proposition}
\newtheorem{conjecture}[theorem]{Conjecture}
\newtheorem{fact}[theorem]{Fact}

\newtheorem{example}[theorem]{Example}
\newtheorem{remark}[theorem]{Remark}
\newtheorem{question}[theorem]{Question}
\newtheorem{problem}[theorem]{Problem}

\newtheorem{claim}{Claim}
\newtheorem{code}[theorem]{GAP-code}
\title{Conjecture of Wilf: a survey}
\author{Manuel Delgado}
\date{\today}
\date{\today}
\address{CMUP, Departamento de Matem\'atica, Faculdade de
	Ci\^encias, Universidade do Porto, Rua do Campo Alegre 687, 4169-007 Porto,
	Portugal} 
\email{mdelgado@fc.up.pt} 
\thanks{The author was partially supported by CMUP (UID/MAT/00144/2019), which is funded by FCT (Portugal) with national (MCTES) and European structural funds (FEDER), under the partnership agreement PT2020, and also by the Spanish project MTM2014-55367-P. Furthermore, the author acknowledges a sabbatical from the FCT: SFRH/BSAB/142918/2018.}

\begin{document}


% 20M14 Commutative semigroups
% 20--02 Group theory and generalizations -- survey articles
% 05--02 Combinatorics -- survey articles
% 11--02 Number theory -- survey articles

\maketitle

\begin{abstract}
	%==================
This paper intends to survey the vast literature devoted to a problem posed by Wilf in 1978 which, despite the attention it attracted, remains unsolved. As it frequently happens with combinatorial problems, many researchers who got involved in the search for a solution thought at some point that a solution would be just around the corner, but in the present case that corner has never been reached.

By writing this paper I intend to give the reader a broad approach on the problem and, when possible, connections between the various available results. With the hope of gathering some more information than just using set inclusion, at the end of the paper a slightly different way of comparing results is developed.
%==================

\end{abstract}

%%%%%%%%%%% introduction %%%%%%%%%%
\section{Introduction}\label{sec:introduction}
% ===========================================
At the beginning, my personal motivation was to build a list of references, each with a summary of the results therein related to Wilf's conjecture. This would have helped me by not having to dive into a collection of papers each time I needed a result. Then I thought that making the list public could also be a contribution to Wilf's conjecture. This process ended up in the writing of this paper, which is in some sense \emph{yet another survey}. Another, because most papers fully dedicated to the conjecture provide good literature reviews. Although not aiming to be complete, these could be taken as surveys.

As it frequently happens with easy to state combinatorial problems, while working on them one thinks that a solution is at reach. This is certainly the case of the problem posed by Wilf, but nevertheless no one has found the aimed solution so far. Tacking into account the number of published papers on the theme, one can infer that much time has globally been dedicated to the problem. This may lead people to classify the problem in the category of dangerous problems, in the sense that one risks to spend too much time struggling with it and have to give up without getting a solution. Fortunately, partial results may be of some interest.

The plan of the paper follows.

This introductory section contains most of the terminology and notation to be used along the paper. 
There are not many differences to what is commonly used. 
This section contains also what I consider a convenient way to visualize numerical semigroups. Although almost all further images appear only in the last section, we provided sufficient information to produce images of semigroups appearing in the remaining parts of the text.

Some problems posed by Wilf are described in the second section, which can be seen as a kind of motivation for the paper.

The third section is the real survey. It contains a large introductory part and then the statements of results grouped into several subsections.

In the final section we introduce a notion of \emph{quasi generalization} (roughly speaking, a set quasi generalizes another if it contains all its elements, except possibly a finite number of them). It allows to draw a lattice involving some important properties that give rise to semigroups satisfying Wilf's conjecture. 
% ===========================================

\subsection{Terminology and notation}\label{sec:notation}
% ===========================================
Most of the notation and terminology used appears in the book by Rosales and García-Sánchez~\cite{RosalesGarcia2009Book-Numerical}. Results referred as well known can be found in the same reference.

Let~$S$ be a numerical semigroup. 
Recall that a numerical semigroup $S$ is a subset of $\mathbb{N}$ (the set of nonnegative integers) such that $0 \in S$, $S$ is closed under addition and the complement $\mathbb{N} \setminus S$ is finite (possibly empty).
Throughout the paper, when the letter $S$ appears and nothing else is said, it should be understood as being a numerical semigroup. 

The \emph{minimal generators} of~$S$ are also known as \emph{primitive elements} of~$S$.
The set of primitive elements of~$S$ is denoted $\primitivesoper(S)$. It is well known to be finite. When there is no possible confusion on which is the semigroup in hand, the notation is often simplified and we write  $\primitivesoper$ instead of $\primitivesoper(S)$. This kind of simplification in the notation is made for all the other combinatorial invariants introduced along the paper.

The \emph{multiplicity} of~$S$ is the least positive integer of~$S$ and is denoted $\multiplicityoper(S)$, or simply~$\multiplicityoper$.
The \emph{Frobenius number} of~$S$ is the largest integer that does not belong to $S$, and is denoted $\Frobeniusoper(S)$.
The \emph{conductor} of~$S$ is simply $\Frobeniusoper(S) +1$. Wilf's notation will be used for the conductor: $\conductoroper(S)$, or simply~$\conductoroper$. Note that $\conductoroper(S)$ is the smallest integer in~$S$ from which all the larger integers belong to~$S$. Let $\qoper(S)=\lceil \conductoroper(S)/\multiplicityoper(S)\rceil$ be the smallest integer greater than or equal to $\conductoroper(S)/\multiplicityoper(S)$. This number is called the \emph{depth} of~$S$ and is frequently denoted just by~$\qoper$. It is worth to keep in mind that $\conductoroper(S)\le \multiplicityoper(S)\qoper(S)$.

The set of \emph{left elements} of~$S$ consists of the elements of~$S$ that are smaller than $\conductoroper(S)$. It is denoted $\leftsoper(S)$ (or simply~$\leftsoper$).
A positive integer that does not belong to $S$ is said to be a \emph{gap} of $S$ (\emph{omitting value} in Wilf's terminology). The cardinality of the set of gaps is said to be the \emph{genus} of $S$ and, following Wilf, is denoted by $\genusoper(S)$, or simply by $\genusoper$. 

If $x\in S$, then $\Frobeniusoper(S)-x\not \in S$. Thus, the following well known remark holds.

\begin{remark}%\cite[Lemma 2.14]{RosalesGarcia2009Book-Numerical}
Let $S$ be a numerical semigroup. Then $\genusoper(S)\ge \conductoroper(S)/2$. 
\end{remark}

As usual, $\lvert X\rvert$ denotes the cardinality of a set $X$.
It is immediate that $\genusoper(S)+\left|\leftsoper(S)\right| =\conductoroper(S)$.
From the above remark it follows that $\conductoroper(S)\ge 2\left|\leftsoper(S)\right|$.

The number of primitives of $S$ is  called the \emph{embedding dimension} of~$S$. As it is just the cardinality of $\primitivesoper(S)$, it can be denoted $\lvert \primitivesoper(S)\rvert $, but in this paper I will mainly use the notation $\embeddingdimensionoper(S)$, or simply~$\embeddingdimensionoper$; $\embeddingdimensionoper$ stands for \emph{dimension} (a short for embedding dimension). 

An integer $x$ is said to be a \emph{pseudo-Frobenius number} of $S$ if $x\not\in S$ and $x+s\in S$, for all $s\in S\setminus{\{0\}}$. The cardinality of the set of pseudo-Frobenius numbers of $S$ is said to be the \emph{type} of $S$ and is denoted by $\typeoper(S)$. The notion of type has been an important ingredient in the discovery of various families of numerical semigroups satisfying Wilf's conjecture, due to Proposition~\ref{prop:type} below.
Another important tool, which is used in a crucial (and frequently rather technical) way in the proofs of some results presented in this survey is the \emph{Apéry set (with respect to the multiplicity)}: $\Aperyoper(S,\multiplicityoper)=\{s\in S\mid s-\multiplicityoper\not\in S\}$.

Let $X$ be a set of positive integers. The notation $\langle X\rangle_{t}$ is used to represent the smallest numerical semigroup that contains $X$ and all the integers greater than or equal to~$t$. 

For a numerical semigroup $S$, the interval of integers starting in $\conductoroper(S)$ and having $\multiplicityoper(S)$ elements is called the \emph{threshold interval} of $S$ (following a suggestion of Eliahou).
% ===========================================

% ===========================================
\subsection{A convenient way to visualize numerical semigroups}\label{subsec:visualization-and-figures}
The pictures in this paper were produced using the \textsf{GAP}~\cite{GAP4-2018} package \textsf{IntPic}~\cite{IntPic-2017}, while the computations have been carried out using the \textsf{GAP} package \textsf{numericalsgps}~\cite{Numericalsgps-2018}.

Let $S$ be a numerical semigroup. The set of nonnegative integers up to $\conductoroper+\multiplicityoper-1$ clearly contains $\leftsoper$ and it is easy to see that it contains $\primitivesoper$ as well. It is helpful to dispose the mentioned integers into a table and to highlight those that, in some sense, are special.

Several figures will be presented to give pictorial views of numerical semigroups. Each of them consists of a rectangular $(\qoper+1)\times \multiplicityoper$-table and the entries corresponding to elements of the semigroup are highlighted in some way. Some gaps can also be emphasized. The entries in uppermost row are those of the threshold interval. 

\begin{example}\label{ex:gens-5-13_cond-20}
	Figure~\ref{fig:gens-5-13_cond-20} is a pictorial representation of the numerical semigroup $\langle 5,13 \rangle_{20}=\langle 5,13,21,22,24 \rangle$. 
	%$\langle 5,13 \rangle_{20}$. 
	The elements of the semigroup are highlighted and, among them, the primitive elements and the conductor are emphasized. When an element is highlighted for more than one reason, gradient colours are used. 
	%This semigroup is of maximal embedding dimension.
	
\thisfloatsetup{floatwidth=.35\hsize,capbesidewidth=sidefil,capposition=beside,capbesideposition=right}
%\thisfloatsetup[figure]{floatwidth=.35\hsize}
	\begin{figure}[h]
		\begin{center}
			\includegraphics[scale=0.8]{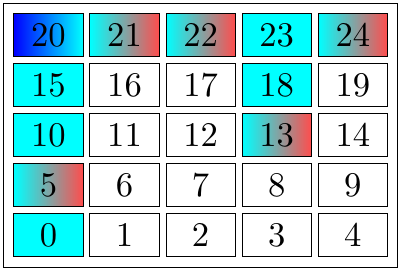}
		\end{center}
		\caption{Pictorial representation of the numerical semigroup $\langle 5,13,21,22,24 \rangle$. \label{fig:gens-5-13_cond-20}}
	\end{figure} 
\end{example}

Observe that there is at most one primitive per column. There is exactly one in each column when the semigroup is  of maximal embedding dimension. 

Note that all the integers in a given column are congruent modulo $\multiplicityoper$. In particular, an element belongs to the Apéry set relative to $\multiplicityoper$ if and only if it is the lowest emphasized element %(i.e., that belongs to $S$) 
in some column (provided that no gaps (for instance the pseudo-Frobenius numbers) are highlighted). 
\medskip

For the benefit of the reader, I explain the way I produced Figure~\ref{fig:gens-5-13_cond-20}, including the \textsf{GAP} code used. 
To start, \textsf{GAP} is taught what my numerical semigroup is (see the manual of \textsf{numericalsgps} for details).
\begin{verbatim}
	ns := NumericalSemigroup(5,13,21,22,24);
\end{verbatim}
Then one can use the following commands to produce the \textsf{TikZ} code for the picture shown (which can be included in a \LaTeX\ document):
\begin{code}\label{code:draw}
\begin{verbatim}
	#cls is given just to make a change to the default colors
	cls := [ "blue","-red","red!70", "black!40" ];
	P := MinimalGenerators(ns);
	m := Multiplicity(ns);
	c := Conductor(ns);
	q := CeilingOfRational(c/m);
	rho := q*m-c;
	list := [-rho .. c+m-1];
	ti := [c..c+m-1]; 
	importants := Union(SmallElements(ns),ti);
	options := rec(colors := cls,highlights:=[[c],importants,P]);
	tkz := IP_TikzArrayOfIntegers(list,m,options);;	
	Print(tkz);
\end{verbatim}
\end{code}
The function \texttt{IP\_TikzArrayOfIntegers} (which produces the \textsf{TikZ} code from the information previously computed using \textsf{numericalsgps}) is part of the \textsf{intpic} package. The manual of the package can be consulted for details and examples. In particular, the manual contains a complete example showing a possible way to include the picture (or its \textsf{TikZ} code) in a \LaTeX\ document.

Executing the following command, the created picture should pop up. As this command depends on some other software, namely the operating system, some extra work on the configuration may be needed.
\begin{verbatim}
IP_Splash(tkz);
\end{verbatim}
If everything goes well, the figure (in \textsf{pdf} format) can be saved and included in the \LaTeX\ document in some standard way.
\begin{example}\label{ex:sgp-with-PF}
	Figure~\ref{fig:sgp-with-PF} is just another example. The following \textsf{GAP} session shows some important data: a numerical semigroup and its pseudo-Frobenius numbers. These are highlighted in the figure, in addition to elements of the semigroup, as in Example~\ref{ex:gens-5-13_cond-20}.
	\begin{verbatim}
	gap> ns := NumericalSemigroup(12, 19, 20, 22, 23, 26, 27, 28, 29);;
	gap> Conductor(ns);
	38
	gap> pf := PseudoFrobenius(ns);
	[ 16, 30, 33, 37 ]	
	\end{verbatim}
	
\thisfloatsetup{floatwidth=.41\hsize,capbesidewidth=sidefil,capposition=beside,capbesideposition=right}
%\thisfloatsetup[figure]{floatwidth=.35\hsize}
	\begin{figure}[h]
		\begin{center}
			\includegraphics[scale=0.8]{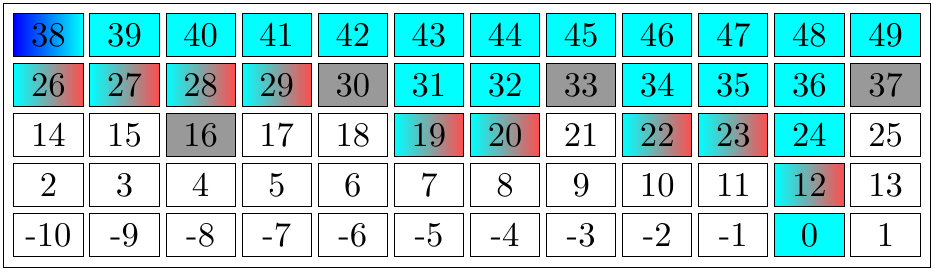}
		\end{center}
		\caption{Pictorial representation of %the numerical semigroup 
			$\langle 12, 19, 20, 22, 23, 26, 27, 28, 29 \rangle$, with the pseudo-Frobenius numbers highlighted. \label{fig:sgp-with-PF}}
	\end{figure} 
\end{example}

The picture can be obtained with just small changes from \textsf{GAP}-code~\ref{code:draw}.
Besides redefining the numerical semigroup, it suffices to replace the line beginning with \textsf{options} by the following two lines of code:
    
\begin{verbatim}
	pf := PseudoFrobenius(ns);;
	options := rec(colors := cls,highlights:=[[c],importants,P,pf]);;
\end{verbatim}

In order to obtain an image just showing the shape, the options can be changed as follows:
\begin{code}\label{code:shape}
\begin{verbatim}# options to produce the shape
	options := rec(
	highlights:=[[],[],[c],importants,[],[],[],[],P],
	cell_width := "6",colsep:="0",rowsep:="0",inner_sep:="2",
	shape_only:=" ",line_width:="0",line_color:="black!20");;	
\end{verbatim}
\end{code}	

\thisfloatsetup{floatwidth=.43\hsize,capbesidewidth=sidefil,capposition=beside,capbesideposition=right}
%\thisfloatsetup[figure]{floatwidth=.35\hsize}
\begin{figure}[h]
	\begin{center}
		\includegraphics[scale=1.8]{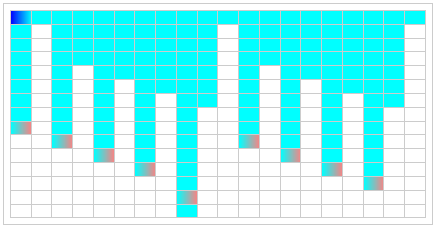}
	\end{center}
	\centering
	\caption{Shape of the semigroup $\langle 20, 49, 58, 67, 76, 85, 94, 103, 112\rangle$ }\label{fig:GAS_m20_h4_d9_l6}
%gap> m := 20;; h := 2;; d := 9;; l := 8;;
%gap> gens := List([1..l], k -> h*m + k*d);
%[ 49, 58, 67, 76, 85, 94, 103, 112 ]
%gap> ns := NumericalSemigroup(Union([m],gens));
%<Numerical semigroup with 9 generators>
%gap> draw_for_W0_shape(ns);	
\end{figure}

% ===========================================

%%%%%%%%%%% Wilf's problems %%%%%%%%%%
\section{Two problems posed by Wilf}\label{sec:problems}
% ===========================================
This section starts with a few words on Wilf's paper~\cite{Wilf1978AMM-circle}, including a transcript of the two problems that Wilf left open. Then, a little about each problem is said.
\subsection{Wilf's paper}
Wilf's concern was: 
to present an algorithm, which, given a numerical semigroup $S$ and a finite generating set for $S$, finds the conductor of $S$, decides whether a given integer is representable (in terms of the elements of the generating set), finds a representation of an element of $S$, and determines the number of omitted values of $S$.

These are problems that many researchers interested in combinatorial problems related to numerical semigroups are nowadays still concerned with. In a somewhat more modern language, one would say that Wilf was concerned with \emph{the Frobenius problem} (see~\cite{Ramirez-Alfonsin2005Book-Diophantine}), the membership problem, factorization problems (see~\cite{GeroldingerHalter-Koch2006}) and the problem of determining the genus.
These problems continue to be (are at the base of) active fields of research. 

The circle of lights algorithm explicitly given in~\cite{Wilf1978AMM-circle} determines both the conductor and the genus of a numerical semigroup provided that a finite generating set is at hand. Wilf also suggests a few changes to the circle of lights algorithm in order test membership and also to find a factorization. He also observed that to test membership, a suggestion of Brauer~\cite{Brauer1942AJM-problem} should be incorporated: it involves the use of the Apéry set (relative to the multiplicity).
Another observation that I would like to make is that the (space and time) complexity is explicitly given, which is a relevant contribution to the overall quality of Wilf's paper. It is extremely agreeable to read and this without doubt contributes to the success of the problems stated in it.

At the end of Wilf's article one finds the following two problems.
It should be understood that the positive integer $k$ represents the embedding dimension of some numerical semigroup. 

\begin{problem}[\cite{Wilf1978AMM-circle}]\label{prob:wilf-a-b}
	Wilf asked:
    \begin{itemize}
        \item[$(a)$] Is it true that for a fixed $k$ the fraction $\Omega/\chi$ of omitted values is at most $1-(1/k)$ with equality only for the generators $k,k+1,\ldots,2k-1$?
        \item[$(b)$] Let $f(n)$ be the number of semigroups whose conductor is $n$. What is the order of magnitude of $f(n)$ for $n\to\infty$?
    \end{itemize}
\end{problem}

The first problem consists in fact of two problems. They can be stated explicitly as follows:
\begin{problem}\label{prob:wilf-a1-a2} Wilf's problem (a) splits into two problems.
	\begin{itemize}
		\item[$(a.i)$] Is it true that for a fixed $k$ the fraction $\Omega/\chi$ of omitted values is at most $1-(1/k)$?
		\item[$(a.ii)$] Is it true that for a fixed $k$ the fraction $\Omega/\chi$ of omitted values is $1-(1/k)$ only for the generators $k,k+1,\ldots,2k-1$?
    \end{itemize}
\end{problem}

Problem ($a.i$) is nowadays known as \emph{Wilf's conjecture}.
Sylvester's result (which is mentioned in the first page of Wilf's paper) gives counter examples to Problem~$(a.ii)$. Apparently Wilf forgot about them. Nowadays there are other counter examples known, but a characterization of those semigroups for which the equality holds is an open problem (see Section~\ref{sec:wilf-a2}).

 %%%%%%%%%%
\subsection{Problem (\emph{a.i}): Wilf's conjecture}\label{sec:wilf-a1}
To a numerical semigroup~$S$ one can associate the following number denoted $\wilfoper(S)$ and called the \emph{Wilf number of~$S$}:
\begin{equation}\label{eq:wilf-number}
    \wilfoper(S) = \lvert \primitivesoper(S) \rvert\lvert \leftsoper(S) \rvert-\conductoroper(S).
\end{equation}

A numerical semigroup is said to be a \emph{Wilf semigroup} if and only if its Wilf number is nonnegative.
\emph{Wilf's conjecture} can be stated as follows:
\begin{conjecture}[Wilf, 1978]\label{conj:wilf}
  Every numerical semigroup is a Wilf semigroup.
\end{conjecture}
It is a simple exercise to verify that Conjecture~\ref{conj:wilf} is precisely Problem~\ref{prob:wilf-a1-a2}~($a.i$).
\medskip

There is another number that can be associated to a numerical semigroup, just as Wilf number is, and which has revealed great importance in recent research (as the reader will be able to confirm, in particular when reading Section~\ref{sec:large-mult}).
Let $S$ be a numerical semigroup and let $D= \primitivesoper(S)\cap \{\conductoroper,\dots,\conductoroper+\multiplicityoper-1\}$ be the set of non primitives in the threshold interval.
Eliahou~\cite{Eliahou2018JEMS-Wilfs} associated to $S$ the number $\eliahouoper(S)$ that appears in Equation~(\ref{eq:eliahou-number}) below and used the notation $\wilfoper_0(S)$ to represent it. I prefer the notation $\eliahouoper(S)$, and use the terminology \emph{Eliahou number of~$S$}:
\begin{equation}\label{eq:eliahou-number}
\eliahouoper(S) = \lvert \primitivesoper\cap \leftsoper\rvert\lvert \leftsoper \rvert - {\qoper} \lvert D\rvert + \qoper\multiplicityoper - \conductoroper.
\end{equation}

Eliahou~\cite[Pg.~2112]{Eliahou2018JEMS-Wilfs} observed that there are numerical semigroups with negative Eliahou number and stated the following problem which is still open.
\begin{problem}
	Give a characterization of the class of numerical semigroups whose Eliahou number is negative.
\end{problem}

%%%%%%%%%%
\subsection{Problem ($\emph{a.ii}$): another open problem}\label{sec:wilf-a2}
A very nice result of Sylvester~\cite{Sylvester1984ET-Mathematical} gives a formula for the Frobenius number of a numerical semigroup of embedding dimension $2$. A closed formula (of a certain type) for the Frobenius number of a numerical semigroup of higher embedding dimension is not at reach for semigroups of bigger embedding dimension (see~\cite{Curtis1990MS-formulas} or \cite[Cor. 2.2.2]{Ramirez-Alfonsin2005Book-Diophantine}). Sylvester's results can be written as follows (see~\cite{RosalesGarcia2009Book-Numerical}): if $S = \langle a, b\rangle$ is a numerical semigroup of embedding dimension $2$, then  $\Frobeniusoper(S) = ab-a-b$, and $\genusoper(S)=\conductoroper(S)/2$. From this, it is immediate that for a numerical semigroup $S$ of embedding dimension $2$, $\wilfoper(S)=0$.
The fact that numerical semigroups of the form $\langle \multiplicityoper, k\multiplicityoper+1,\ldots k\multiplicityoper+\multiplicityoper-1\rangle$ (which are of maximal embedding dimension and generated by some generalized arithmetic sequences) have Wilf number equal to $0$ is straightforward (see~\cite{FroebergGottliebHaeggkvist1987SF-numerical,MoscarielloSammartano2015MZ-conjecture}). 

Whether these are the only numerical semigroups for which Wilf number is $0$ is a slight modification of Problem~\ref{prob:wilf-a1-a2}($a.ii$) and is open. I rephrase the question stated by  Moscariello and Sammartano.

\begin{problem}[{\cite[Question~8]{MoscarielloSammartano2015MZ-conjecture}}]\label{prob:wilf-equality}
	\label{quest:W-equal-0}
	Let $S=\langle \multiplicityoper, g_2,\ldots,g_{\embeddingdimensionoper}\rangle$ be a numerical semigroup with multiplicity~$\multiplicityoper$ and embedding dimension~$\embeddingdimensionoper$. Is it true that if $\wilfoper(S)=0$, then $\embeddingdimensionoper(S)=2$ or $\embeddingdimensionoper(S)=\multiplicityoper(S)$ and there exists an integer $k>1$ such that $g_i=k\multiplicityoper+(i-1)$, for $i\in\{2,\ldots,\embeddingdimensionoper\}$? 
\end{problem}

Moscariello and Sammartano observed that in order to answer affirmatively this question it suffices to prove that for a semigroup with Wilf number equal to $0$, either its embedding dimension is $2$ or it has maximal embedding dimension.

They also observed that no numerical semigroup of genus up to $35$ provides a negative answer to the question.

Kaplan~\cite[Prop.~26]{Kaplan2012JPAA-Counting} has shown that Problem~\ref{prob:wilf-equality} has a positive answer in the case of numerical semigroups whose multiplicity is at least half of the conductor. The same holds for numerical semigroups of depth $3$ (see a remark by Sammartano in~\cite[Rem.~6.6]{Eliahou2018JEMS-Wilfs}), thus concluding that there are no counter examples among the semigroups satisfying $\conductoroper\le 3\multiplicityoper$.

%%%%%%%%%%%%
\subsection{Problem (b): counting numerical semigroups}

Wilf's Problem~\ref{prob:wilf-a-b}($b$) can be viewed as a problem about counting numerical semigroups by conductor. Backelin~\cite{Backelin1990MS-number} addressed this problem. A slight modification consists on counting by genus. Great attention has been given to this problem after Bras-Amorós~\cite{Bras-Amoros2008SF-Fibonacci} proposed some conjectures on the theme. Some of these conjectures were solved by Zhai~\cite{Zhai2012SF-Fibonacci}, while others remain open.
For an excellent survey (which in particular contains references for counting by conductor and has an outline of Zhai's proofs), see Kaplan~\cite{Kaplan2017AMM-Counting}.

Denote respectively by $N(g)$ and $t(g)$ the number of numerical semigroups of genus $g$ and the number of numerical semigroups of genus $g$ satisfying $\conductoroper(S) \le 3\multiplicityoper(S)$.

In the paper where he proved some of the conjectures of Bras-Amorós (one of them being that the sequence $\big(N(g)\big)$ behaves like the Fibonacci sequence), Zhai also proved that the proportion of numerical semigroups such that $\conductoroper(S) \le 3\multiplicityoper(S)$ tends to $1$ as $g$ tends to infinity, as conjectured by Zhao~\cite{Zhao2010SF-Constructing}.

\begin{proposition}[\cite{Zhai2012SF-Fibonacci}]\label{prop:Zhao-Zhai}
	With the notation introduced, the following holds:
	$$\lim_{g\to\infty} \frac{t(g)}{N(g)}=1.$$
\end{proposition}

I will leave here a question that can be stated in a similar way to Zhao's conjecture. It will be better appreciated when reading Section~\ref{sec:large-emb-dim} (and confronting with Section~\ref{sec:large-mult}).

Denote by $p(g)$ the number of numerical semigroups of genus $g$ satisfying $\embeddingdimensionoper(S) \ge \multiplicityoper(S)/3$.
\begin{question}\label{quest:emb-dim-asymptotics}
	Does $\lim_{g\to\infty} \frac{p(g)}{N(g)}$ exist?
\end{question}
% ===========================================

%%%%%%%%%%% Particular cases %%%%%%%%%%
\section{Some classes of Wilf semigroups}\label{sec:particular-cases}
% ===========================================
As already observed, it follows from a result of Sylvester that semigroups of embedding dimension $2$ have Wilf number equal to $0$. 
In particular, semigroups of embedding dimension $2$ are Wilf semigroups. 
Many other classes are known to consist of Wilf semigroups. This section gives an account of a large number of them.

The theme is rather popular and it is frequent to check a family of numerical semigroups against Wilf's conjecture, whenever that new family of numerical semigroups is investigated for some possibly other reason.
It may well happen that some results are not referred to in this paper. This is far from meaning that I do not consider the ideas involved important.
In a few cases this may be a matter of choice, but most probably it simply means that the results are not part of my very restricted knowledge. For that, I humbly express my apologies both to the authors and the readers.
 
The results are split into several subsections, according to a criterion that seems difficult to explain. It finally just aims at putting results together so that they can be compared with ease.

Most of the families considered are described through at least two combinatorial invariants such as embedding dimension, the multiplicity or the conductor.
Exceptions (besides finite sets) are families that are completely described by using only one invariant among the embedding dimension, the multiplicity or the number of left elements.
 
Inside each subsection several results are mentioned (through precise numbered statements or just in the text) and there are cases in which the most general one is stated as a theorem. In a few cases there are results mentioned in more than one subsection.
 
In the first subsection, there is an emphasis on a particular ingredient used along the proofs of the various results. The ingredient is an inequality involving the type. As illustrations of the results that can be found there, we shall see that semigroups with embedding dimension up to three, almost symmetric numerical semigroups and those semigroups generated by generalized arithmetic sequences are Wilf semigroups. 
 
The second subsection is about families of semigroups (somehow explicitly) given by some sets of generators. The examples therein share the particularity that all the members have negative Eliahou number.

The third subsection refers to numerical semigroups with nonnegative Eliahou number.
 
The fourth subsection is dedicated to constructions that are somehow natural. In fact, only one such construction is given here: dilations of numerical semigroups. This subsection could certainly be filled with other constructions. My choice just reflects the feeling that possible generalizations could be worth exploring.

Then there is a subsection devoted to numerical semigroups of small multiplicity.

The sixth subsection is about results in which the main attention is given to semigroups with large embedding dimension, when compared to the multiplicity.

Next there appears a subsection containing a result involving numerical semigroups with big multiplicities and possibly small embedding dimensions.
 
The eighth subsection is similar to the sixth, but now the results have an emphasis on semigroups with large multiplicity, when compared to the conductor.
 
In ninth subsection there is a result taking into account an invariant not previously considered (at least in a fundamental way, to the best of my knowledge). It is the second smallest primitive, sometimes called the \emph{ratio}.
 
The final subsection is concerned with families of numerical semigroups that can be described using only one combinatorial invariant. 

%%%%%%%%%%
\subsection{The type as an important ingredient}

The following proposition, due to Fröberg, Gottlieb and Haeggkvist, is at the base of some results on Wilf's conjecture. It implies that semigroups whose type is smaller than its embedding dimension are Wilf.

\begin{proposition}[{\cite[Theorem 20]{FroebergGottliebHaeggkvist1987SF-numerical}}]\label{prop:type}
    Let $S$ be a numerical semigroup. Then $\conductoroper(S)\le(\typeoper(S)+1)\left|\leftsoper(S)\right|$.
\end{proposition}

By proving that the type of a numerical semigroup of embedding dimension~$3$ is either~$1$ or~$2$ (\cite[Th.~11]{FroebergGottliebHaeggkvist1987SF-numerical}) and using the fact that $\conductoroper\ge 2\left|\leftsoper\right|$ referred in Section~\ref{sec:notation}, they obtained that numerical semigroups of embedding dimension~$3$ are Wilf, a result that Dobbs and Matthews~\cite[Cor.~2.6]{DobbsMatthews2006} reproved using a different approach.
Using Sylvester's result for embedding dimension~$2$ and the fact that~$\mathbb{N}$ is Wilf, the same authors obtained the following result.

\begin{theorem}[{\cite[Th.~20]{FroebergGottliebHaeggkvist1987SF-numerical}, \cite[Th.~2.11]{DobbsMatthews2006}}]\label{th:type<=3}
    Numerical semigroups of embedding dimension smaller than $4$ are Wilf. 
\end{theorem}

Since there is no upper bound for the type of numerical semigroups of embedding dimension bigger than $3$, (see~\cite[pg.~75]{FroebergGottliebHaeggkvist1987SF-numerical}, for an example due to Backelin), Proposition~\ref{prop:type} can not be used to obtain other general results (but can, and has been applied successfully to particular families of semigroups).

A numerical semigroup $S$ is said to be \emph{irreducible} if it cannot be expressed as the intersection of
two numerical semigroups properly containing it.
$S$ is said to be \emph{symmetric} if it is irreducible and $\Frobeniusoper(S)$ is odd and it is said to be \emph{pseudo-symmetric} if it is irreducible and $\Frobeniusoper(S)$ is even.
One could take the following as definition (see~\cite[Cor~4.5]{RosalesGarcia2009Book-Numerical}): $S$ is \emph{symmetric} if and only if $\genusoper(S) = \conductoroper(S)/2$, while $S$ is \emph{pseudo-symmetric} if and only if $\genusoper(S) = (\conductoroper(S)+1)/2$. 

It can be proved as a simple exercise that irreducible numerical semigroups are Wilf. A more involving proof could be to observe that the type of this class of semigroups does not exceed $2$. 
\begin{proposition}[{\cite[Prop.~2.2]{DobbsMatthews2006}}]\label{wilf-particular:irreducible}
    Irreducible numerical semigroups are Wilf.
    
\end{proposition}
The above result was generalized by Marco La Valle in~\cite[Th. 5.5]{Barucci2009}. Before stating this generalization, a further definition is needed.
A numerical semigroup is said to be \emph{almost symmetric} if its genus is the arithmetic mean of its Frobenius number and its type (see~\cite{BarucciFroeberg1997JA-One-dimensional}). %ok (the same as t=(g+1-2n)+1)
It is a class of semigroups that includes the symmetric ones and the pseudo-symmetric. 

\begin{proposition}[{\cite[Th.~5.5]{Barucci2009}}]\label{prop:wilf-almost-symmetric}
	Almost symmetric numerical semigroups are Wilf.
\end{proposition}

As a consequence of Theorem~\ref{th:type<=3} Dobbs and Matthews derived an interesting corollary:
\begin{corollary}[{\cite[Cor~2.7]{DobbsMatthews2006}}]\label{cor:c<4l} 
	If $S$ is a numerical semigroup with $\conductoroper\le 4 \left|\leftsoper\right|$, then $S$ is Wilf.
\end{corollary}

Also making use of Proposition~\ref{prop:type}, Kunz~\cite{Kunz2016SF-type} obtained the following result (for $p$ and $q$ coprime). See also Kunz and Waldi~\cite{KunzWaldi2017JA-deviation} for some other generalizations.

\begin{proposition}[{\cite[Cor.~3.1]{KunzWaldi2017JA-deviation}}]
	Let $S$ be a numerical semigroup with $\embeddingdimensionoper(S)\ge 3$. Let $p$ and $q$ be two distinct primitives of $S$.
	If $g+h\in (p+S)\cup(q+S)$, for any (non necessarily distinct) primitives $g$ and $h$ of $S$, then $\typeoper(S)\le\embeddingdimensionoper(S) - 1$. In particular, $S$ is Wilf.
\end{proposition}

A semigroup generated by a generalized arithmetic sequence is a semigroup of the form $S = \langle m, hm + d, hm + 2d, \ldots, hm + \ell d\rangle$, where $m, d, h, \ell$ are positive integers such that $m\ge 2$, $\gcd(m,d) = 1$ and $\ell \le m-2$. Note that $m$ and $d$ being coprime ensures that $S$ is a numerical semigroup.
For a picture made up from a semigroup generated by a generalized arithmetic sequence with $m= 20, d=9, h=2, \ell=8$, see Figure~\ref{fig:GAS_m20_h4_d9_l6}.
By using a result of Matthews~\cite[Cor.~3.4]{Matthews2005I-integers} that computes the type of a numerical semigroup generated by a generalized arithmetic sequence, Sammartano observed the following:
\begin{proposition}[{\cite[Prop.~20]{Sammartano2012SF-Numerical}}]\label{wilf-particular:generalized-arithmetic}
	Numerical semigroups generated by generalized arithmetic sequences are Wilf.
\end{proposition}

%%%%%%%%%%%%
\subsection{Semigroups given by sets of generators}\label{sec:sgps-given-sets-generators}

Let $G$ be an abelian group. Let $A\subseteq G$ be a non empty finite subset and let $h$ be a positive integer. (For the purpose of this paper the reader may think of the group as being a cyclic group $\mathbb{Z}/m$ and take $h=3$; for more details, see~\cite[Chap.~4]{TaoVu2006Book-Additive}.)

The set $A$ is said to be a \emph{$B_h$ set} if, for all $a_1,\ldots, a_h,b_1,\ldots, b_h\in A$, the equality
\[a_1 + \cdots + a_h = b_1 + \cdots + b_h\]
holds if and only if $(a_1,\ldots, a_h)$ is a permutation of $(b_1,\ldots, b_h)$. 

Let $m,a,b,n\in\mathbb{N}_{>0}$ be such that $n\ge 3$ and 
\[(3m+1)/2\le a< b \le (5m-1)/3.\]

Let $A\subseteq \{a,\ldots,b\}$ be such that $\left| A \right|=n-1$ and $A$ induces a $B_3$ set in $\mathbb{Z}/m$. That such a set exists follows from~\cite[Proposition 3.1]{EliahouFromentin2018SF-misses}. Finally, let 
\[S=\langle \{m\}\cup A\rangle_{4m}.\]

Eliahou and Fromentin proved the following result:
\begin{proposition}[{\cite[Th.~4.1]{EliahouFromentin2018SF-misses}}]\label{prop:misses}
	Let $S=\langle \{m\}\cup A\rangle_{4m}$ be a semigroup as constructed above. Then
$\wilfoper(S)\ge 9$, in particular $S$ is a Wilf semigroup.
\end{proposition}

Let $p$ be an even positive integer, let $\mu=\mu(p) = \frac{p^2}{4}+2p+2$ and let $\gamma=\gamma(p) = 2\mu(p) - \left(\frac{p}{2} +4\right)$. The following holds:

\begin{proposition}[{\cite[Prop.6]{Delgado2018MZ-question}}]\label{prop:S(q)}
	Let $S=S(p)=\langle \mu,\gamma,\gamma+1\rangle_{p\mu}.$ Then
$\wilfoper(S)>0$, in particular, $S$ is a Wilf semigroup.
\end{proposition}

\begin{example}\label{ex:fromentin1}
	Figure~\ref{fig:fromentin1} is a pictorial representation of the numerical semigroup $S(4)$.
	Note that $S(4)$ is of the form $\langle \{m\}\cup A\rangle_{4m}$, with $m=14$ and $A=\{22,23\}$.
%	\thisfloatsetup{floatwidth=.48\hsize,capbesidewidth=sidefil,capposition=beside,capbesideposition=right}
	\begin{figure}[h]
		\begin{center}
			\includegraphics[scale=0.8]{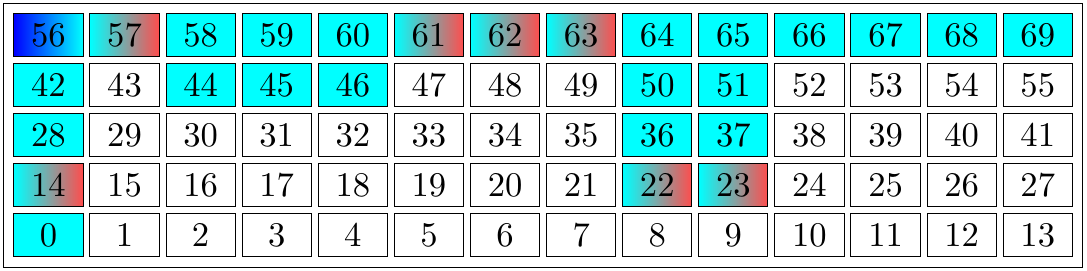}
		\end{center}
		\caption{Pictorial representation of $S(4)=\langle 14,22,23 \rangle_{56}$. \label{fig:fromentin1}}
	\end{figure} 
\end{example}

To end this subsection I would like to make the following observations:
\begin{remark}\label{rem:non-eliahou}
	The numerical semigroups $S=\langle \{m\}\cup A\rangle_{4m}$ and $S(p)=\langle \mu,\gamma,\gamma+1\rangle_{p\mu}$ defined above have (possibly large) negative Eliahou numbers (see \cite{Delgado2018MZ-question,EliahouFromentin2018SF-misses}). The proof that they are Wilf involves explicit counting.
\end{remark}
\begin{remark}
	The semigroups $\langle \{m\}\cup A\rangle_{4m}$ have depth $4$, while, since the conductor of $S(p)$ is $p\mu$,  there is no bound for the depths of the  semigroups $S(p)$.   
\end{remark}
\begin{remark}
	Several other families obtained using similar constructions to the one in Proposition~\ref{prop:S(q)} can be found in the same paper. In particular, for any given integer $n$,  an infinite family of numerical semigroups with Eliahou number equal to $n$ is obtained.
	All these families consist entirely of Wilf semigroups.
\end{remark}
   
%%%%%%%%%%%
\subsection{Semigroups with nonnegative Eliahou numbers}

Recall that the Eliahou number of a numerical semigroup was introduced in page~\pageref{eq:eliahou-number}. The following result, which states that semigroups with nonnegative Eliahou number are Wilf, appears in~\cite[Prop.~3.11]{Eliahou2018JEMS-Wilfs} (see also~\cite[Cor.~2.3]{EliahouFromentin2018SF-misses}).

\begin{proposition}\label{prop:eliahou-implies-wilf}
	Let $S$ be a numerical semigroup with $\eliahouoper(S)\ge 0$. Then $\wilfoper(S)\ge~0$.
\end{proposition}

This result is similar to Proposition~\ref{prop:type} in the sense that in order to prove that a numerical semigroup is Wilf it suffices to prove that it has nonnegative Eliahou number. The main consequences are referred to in Section~\ref{sec:large-mult}.

While waiting for those consequences, let me refer a result obtained by Eliahou and Marín-Aragón. As they observed, the number $12$ that appears in the statement is the best that can be obtained in this way: Example~\ref{ex:fromentin1} gives a counter example for $\left|\leftsoper\right|= 13$.

\begin{proposition}[\cite{EliahouMarin-Aragon2019}]\label{prop:L<=12}
	If $S$ is a numerical semigroup with $\left|\leftsoper\right|\le 12$, then $S$ has nonnegative Eliahou number and therefore is Wilf. 
\end{proposition}

%%%%%%%%%%%
\subsection{Natural constructions}
Barucci and Strazzanti gave in~\cite{BarucciStrazzanti2018SF-Dilatations} the definition of \emph{dilation of $S$ with respect to $a$}. Namely $D(S,a)=\{0\}\cup \{S+a\mid s\in S\setminus\{0\}\}$. Moreover they proved the following result.
\begin{proposition}[\cite{BarucciStrazzanti2018SF-Dilatations}]\label{prop:dilations}
	If $S$ is a Wilf semigroup $S$ then so is any dilation of $S$.
\end{proposition}

Thus, for each Wilf semigroup $S$, the class $\{D(S,a)\mid a\in S\}$ is an infinite family of Wilf semigroups.

\subsection{Semigroups with small multiplicity}\label{sec:small-multiplicity}
Sammartano~\cite[Cor.~19]{Sammartano2012SF-Numerical} proved that semigroups of multiplicity not greater than $8$ are Wilf. Eliahou~\cite{Eliahou2017-Umea} announced the same kind of result but for multiplicity $12$ (with a similar proof; see the comment just after Theorem~\ref{th:large-embedding-dim-3}). In the meantime, Dhany~\cite[Cor.~4.10]{Dhayni2018PJM-Wilfs} had obtained the result for multiplicity $9$.

A big breakthrough was obtained by Bruns, García-Sánchez and O'Neil, who achieved multiplicity $16$. Their proof involves computational methods, combined with geometrical ones, such as the use of Kunz polytopes.  
\begin{proposition}[\cite{BrunsGarcia-SanchezONeill2019}]\label{prop:mult-16}
	Let $S$ be a numerical semigroup with $\multiplicityoper\le 16$. Then $S$ is Wilf.
\end{proposition}

%%%%%%%%%%%%
\subsection{Semigroups with large embedding dimension (compared to the multiplicity)}\label{sec:large-emb-dim}
As stated in Theorem~\ref{th:type<=3}, numerical semigroups with very small ($\le 3$) embedding dimension are Wilf.
The same happens for those with large embedding dimension (when compared to the multiplicity). 
There are several proofs of the fact that semigroups of maximal embedding dimension (i.e., with embedding dimension equal to the multiplicity). The first one (to the best of my knowledge) is due to Dobbs and Matthews~\cite[Cor.~2.4]{DobbsMatthews2006}.
Sammartano~\cite[Th.~18]{Sammartano2012SF-Numerical}, by means of a rather technical proof involving Apéry sets and the counting in intervals of length $\multiplicityoper$ of elements in the semigroup, proved that numerical semigroups with embedding dimension at least half the multiplicity are Wilf. With refined arguments, 
Dhayni 
(see also her thesis~\cite[Th. 2.3.13]{Dhayni2017phd-Problems}) 
generalized Sammartano's result.

\begin{proposition}[{\cite[Th. 4.12]{Dhayni2018PJM-Wilfs}}]\label{wilf-particular:large-embedding-dim}
    Let $S$ be a numerical semigroup with $\left(2+\frac{1}{\qoper}\right)\embeddingdimensionoper\ge \multiplicityoper$. Then $S$ is Wilf.
\end{proposition}

Eliahou obtained the following impressive generalization by using a graph theoretical approach. The concept of matching (set of independent edges) in a certain graph associated to the Apéry set is used. 

\begin{theorem}[\cite{Eliahou2017-Umea}]\label{th:large-embedding-dim-3}
	Let $S$ be a numerical semigroup with $3\embeddingdimensionoper(S)\ge \multiplicityoper(S)$. Then $S$ is Wilf.
\end{theorem}

\subsubsection*{Some comments}
\begin{enumerate}
    \item Previous result implies that semigroups of multiplicity up to $12$ are Wilf, as already mentioned. 
    Note that a non Wilf semigroup $S$ must satisfy $\embeddingdimensionoper(S)\ge 4$ by Theorem~\ref{th:type<=3}. As it must also satisfy $3\embeddingdimensionoper(S)<\multiplicityoper(S)$, it follows that if $S$ is non Wilf, then $\multiplicityoper(S) > 12$. (Sammartano's result implies $\multiplicityoper(S) > 8$ for $S$ non Wilf.)
    \item A positive answer to Question~\ref{quest:emb-dim-asymptotics} would lead to possibly interesting consequences. For instance, if the limit were $1/2$, one could conclude that asymptotically, as $\genusoper$ grows, half of the numerical semigroups satisfy $3\embeddingdimensionoper\ge \multiplicityoper$, and consequently are Wilf.
\end{enumerate}

Two other simple consequences of Theorem~\ref{th:large-embedding-dim-3} follow. The first is a generalization of a result of Dhany~\cite[Th.~4.9]{Dhayni2018PJM-Wilfs}, who proved that semigroups satisfying $\multiplicityoper-\embeddingdimensionoper\le 5$ are Wilf. At this stage, this remark does not give anything new, but gives a very explicit result.
\begin{remark}
	If $\embeddingdimensionoper(S)\ge\multiplicityoper(S)-10$, then $S$ is Wilf.
\end{remark}
\begin{proof}
	If $\multiplicityoper(S)\le 16$, then Proposition~\ref{prop:mult-16} can be used.
	If $\multiplicityoper(S)\ge 16$, then $\multiplicityoper(S)/3 > 5$. But in this case $\multiplicityoper(S) - 10 \ge \multiplicityoper(S)/3$. In fact, this last inequality is equivalent to $2\multiplicityoper(S) \ge 30$, which holds by hypothesis.
\end{proof}

Another consequence (it suffices Sammartano's result to get it), was observed by Eliahou.
\begin{proposition}[{\cite[Prop.~7.6]{Eliahou2018JEMS-Wilfs}}]
	Let $S$ be a numerical semigroup with $\gcd(\leftsoper\cap \primitivesoper) \ge~2$. Then $S$ is Wilf.
\end{proposition}

%%%%%%%%%%%
\subsection{Semigroups with big multiplicity (and possibly small embedding dimension)}
Moscariello and Sammartano proved that for every fixed value of $\lceil \multiplicityoper/\embeddingdimensionoper\rceil$ the conjecture holds for all values of $\multiplicityoper$ which are sufficiently large and are not divisible by a finite set of primes. Recall from previous subsection that the cases $\lceil \multiplicityoper/\embeddingdimensionoper\rceil\le 3$ have been solved. 
\begin{proposition}[{\cite[Th.~1]{MoscarielloSammartano2015MZ-conjecture}}]\label{prop:very-large-mult}
	Let $S$ be a numerical semigroup. Let $\rho=\big\lceil \frac{\multiplicityoper(S)}{\embeddingdimensionoper(S)}\big\rceil$ and let $\phi$ be the product of prime factors of $\rho$. If $\rho>3$, $\multiplicityoper(S)\ge\frac{\rho(3\rho^2-\rho-4)(3\rho^2-\rho-2)}{8(\rho-2)}$ and $\gcd(\multiplicityoper(S),\phi)=1$, then $S$ is Wilf. 
\end{proposition}
The multiplicities of the semigroups that arise from Proposition~\ref{prop:very-large-mult} are large, as the following \textsf{GAP} session suggests (by showing that for $\rho=4$ the smallest multiplicity is $1680$).
\begin{verbatim}
gap> mult := r -> (r*(3*r^2-r-4)*(3*r^2-r-2))/8*(r-2);
function( r ) ... end
gap> mult(4);
1680
\end{verbatim}

%%%%%%%%%%%
\subsection{Semigroups with large multiplicity (compared to the conductor)}\label{sec:large-mult}
It was proved by Kaplan~\cite[Th.~24]{Kaplan2012JPAA-Counting} that numerical semigroups with conductor not greater than twice the multiplicity are Wilf. This result was generalized by Eliahou. One of the ingredients he used is a theorem of Macaulay on the growth of Hilbert functions of standard graded algebras.
In fact, he proved the following:

\begin{proposition}[{\cite[Th.~6.4]{Eliahou2018JEMS-Wilfs}}]\label{prop:large-mult-Eliahou-jems}
	Let $S$ be a numerical semigroup with $\conductoroper(S)\le 3\multiplicityoper(S)$. Then $S$ has nonnegative Eliahou number.
\end{proposition}
This result combined with Proposition~\ref{prop:eliahou-implies-wilf} leads to the following major fact. (Its importance can be better appreciated by seeing the comments that follow the statement.)

\begin{theorem}[{\cite[Cor.~6.5]{Eliahou2018JEMS-Wilfs}}]\label{th:large-mult-Wilf-jems}
  Let $S$ be a numerical semigroup with $\conductoroper(S)\le 3\multiplicityoper(S)$. Then $S$ is Wilf.
\end{theorem}
\subsubsection*{Some comments}
\begin{enumerate}
	\item Denote by $e(g)$ the number of numerical semigroups of genus $g$ having positive Eliahou number. Combining Proposition~\ref{prop:large-mult-Eliahou-jems} with Zhai's Proposition~\ref{prop:Zhao-Zhai}  one sees that $\lim_{g\to\infty} \frac{e(g)}{N(g)}=1.$ Thus, in the sense given by this limit, one can say that asymptotically, as the genus grows, all numerical semigroups have nonnegative Eliahou number. Consequently, asymptotically, as the genus grows, all numerical semigroups are Wilf.
    \item I observe that, despite this asymptotic result concerning Eliahou numbers, there are infinitely many numerical semigroups with negative Eliahou number (see \cite{Delgado2018MZ-question,EliahouFromentin2018SF-misses}). All the examples given in the mentioned papers are Wilf semigroups (some of them appear in Section~\ref{sec:sgps-given-sets-generators}).
\end{enumerate}

%%%%%%%%%%%
\subsection{Considering unusual invariants}\label{sec:ratio}

The second smallest primitive is sometimes called the \emph{ratio} (see~\cite[Exercise~2.12]{RosalesGarcia2009Book-Numerical}).

Spirito proved that if the ratio is large and the multiplicity is bounded by a quadratic function of the embedding dimension, then $S$ is Wilf. He also proved various related statements. As an illustration, I choose one that is rather explicit:

\begin{proposition}[{\cite[Prop~4.6]{Spirito2017ae-Wilfs}}]\label{prop:ratio}
	Let $S$ be a numerical semigroup with ratio $r$ and embedding dimension $\embeddingdimensionoper \ge 10$. If
	\begin{equation}\label{eq:prop-ratio}
	r> \frac{\conductoroper(S)+\multiplicityoper(S)}{3} \text{ and } \multiplicityoper(S)\le \frac{8}{25}\embeddingdimensionoper^2+\frac{1}{5}\embeddingdimensionoper-\frac{5}{4}
	\end{equation}
	then $S$ is a Wilf semigroup.
\end{proposition}
\begin{remark}
	It is straightforward to check that that if $d\le 9$, then $8/25d^2 + 1/5d - 5/4\le 3d$. The following \textsf{GAP} session may help to quickly convince the reader:
	\begin{verbatim}
		gap> f := d -> 8/25*d^2 + 1/5*d - 5/4;
		gap> Int(f(9));
		26
	\end{verbatim} 
	One concludes by using Theorem~\ref{th:large-embedding-dim-3} that the restriction $\embeddingdimensionoper \ge 10$ can be removed from the statement of Proposition~\ref{prop:ratio}.	Moreover, when $\embeddingdimensionoper < 10$ there is no need to impose any restriction on the ratio.
\end{remark}

%%%%%%%%%%%
\subsection{Families described through one invariant}
%%%%%%%%%%%%
Families of semigroups described by limiting the multiplicity of its members were already considered in Section~\ref{sec:small-multiplicity}. Proposition~\ref{prop:mult-16} could have been stated in this subsection, as well as Proposition~\ref{prop:L<=12} which refers to the number of left elements.

Dobbs and Matthews~\cite[Th.~2.11]{DobbsMatthews2006} proved that numerical semigroups with $\left|\leftsoper\right|\le 4$ are Wilf semigroups. As a corollary they obtained that semigroups with $\conductoroper\le 21$ are Wilf. Eliahou~\cite[Prop.~7.4]{Eliahou2018JEMS-Wilfs} observed that numerical semigroups with less than $7$ left elements are Wilf.
These results have been largely superseded. 

Recall that Proposition~\ref{prop:L<=12} gives a similar result, but the restriction on the number of left elements was weakened: numerical semigroups with $\left|\leftsoper\right|\le 12$ are Wilf.
\medskip

Let $S$ be a non Wilf semigroup. By Proposition~\ref{prop:mult-16}, $\multiplicityoper(S)\ge 17$. Using Proposition~\ref{prop:large-mult-Eliahou-jems}, which guarantees that non Wilf semigroups satisfy $\conductoroper(S)>3\multiplicityoper(S)$, one gets that $\conductoroper(S)>54$. 
This proves that semigroups with conductor smaller than $53$ are Wilf. 
\medskip

Fromentin and Hivert, through exhaustive computation, have shown that there are no non Wilf semigroups with genus smaller than $61$. The previous published record, genus $51$, had been obtained by Bras-Amorós\cite{Bras-Amoros2008SF-Fibonacci}.

\begin{theorem}[\cite{FromentinHivert2016MC-Exploring}]\label{th:wilf-particular-genus-60}
Every numerical semigroups of genus up to $60$ is Wilf.
\end{theorem}
Since the genus of a numerical semigroup is not smaller than its conductor plus one, the following consequence, which supersedes the above results concerning the conductor, is immediate.
\begin{corollary}
	Semigroups whose conductor does not exceed $61$ are Wilf.
\end{corollary}

(I am currently developing techniques to replace in Theorem~\ref{th:wilf-particular-genus-60} the integer $60$ by a larger one.
It will probably be part of an experimental preprint of mine which is in an advanced phase of preparation 
and is provisionally entitled ''Wilf's conjecture on numerical semigroups holds for small genus''.)
% ===========================================

\section{Quasi-generalization}\label{sec:quasi-generalization}
% ===========================================
Denote by $\mathfrak{S}$ the class of all numerical semigroups. Let $\mathfrak{W}=\{S\in\mathfrak{S}\mid \wilfoper(S)\ge 0\}$ and let $\mathfrak{E}=\{S\in\mathfrak{S}\mid \eliahouoper(S)\ge 0\}$. 

Whether $\mathfrak{S}=\mathfrak{W}$ is presently not known (Wilf's conjecture says that the equality holds, but it is still a conjecture).
That $\mathfrak{E}\subseteq \mathfrak{W}$ follows from Proposition~\ref{prop:eliahou-implies-wilf}, and up to genus $60$ there are exactly $5$ numerical semigroups not in $\mathfrak{E}$, thus showing that the inclusion is strict (the examples were obtained by Fromentin and appear in~\cite[pgs~2112,2113]{Eliahou2018JEMS-Wilfs}).
The following is a consequence of Remark~\ref{rem:non-eliahou}.
\begin{fact}\label{fact:infinite-non-eliahou}
	$\mathfrak{W}\setminus \mathfrak{E}$ is infinite.
\end{fact}

Let $\mathcal{P}$ be a property (about numerical semigroups). For instance, ``$\embeddingdimensionoper\ge 3$'' is such a property.
Let $\mathfrak{P}=\{S\in\mathfrak{S}\mid S\models \mathcal{P}\}$ be the class of numerical semigroups satisfying $\mathcal{P}$. With this notation, most results in the previous sections can be written in the following form: ``If $S\in\mathfrak{P}$, then $S\in \mathfrak{W}$.'', or ``If $S$ satisfies $\mathcal{P}$, then $S$ is Wilf''. 

I invite the reader to think an all the results as if they had been written in this form. Some properties cannot be as nicely written as in the above example (``$\embeddingdimensionoper\ge 3$''). However, for instance, the statement ``We say that $S$ satisfies property $\mathcal{P}$ if and only if $S$ is of the form $S(p)$, with $p$ an even positive integer.'' allows to write Proposition~\ref{prop:S(q)} in the above form.

By doing so, one can associate a property to each result and conversely. 
Although I do not intend to explicitly give names to all the properties corresponding to the results stated, there are some exceptions:
\begin{itemize}[leftmargin=*]
	\item[]$\mathbf{\mathcal{D}_3}$ stands for the property ``$\embeddingdimensionoper\ge 3$'', which is associated to Theorem~\ref{th:type<=3}. The corresponding class of semigroups is $\mathfrak{D}_3$. 
\end{itemize} 
	Similarly, one has the correspondences:
\begin{itemize}[noitemsep,leftmargin=*]
	\item[]$\mathbf{\mathcal{D}}$ --- ``$3\embeddingdimensionoper\ge \multiplicityoper$'' --- Theorem~\ref{th:large-embedding-dim-3} --- $\mathfrak{D}$;
	\item[]$\mathbf{\mathcal{M}}$ --- ``$\conductoroper\le 3\multiplicityoper$'' --- Theorem~\ref{th:large-mult-Wilf-jems} --- $\mathfrak{M}$;
	\item[]$\mathbf{\mathcal{G}_{60}}$ --- ``$\genusoper \le 60$'' --- Theorem~\ref{th:wilf-particular-genus-60} --- $\mathfrak{G}_{60}$.
\end{itemize} 

It seems reasonable to add other exceptions: $\mathcal{S},\mathcal{W}, \mathcal{E}$ are the properties about numerical semigroups associated, respectively, to $\mathfrak{S},\mathfrak{W}, \mathfrak{E}$. (Note that $\mathcal{S}$ is trivial: it is satisfied by all numerical semigroups.)

\begin{fact}
	All the classes $\mathfrak{D}_3,\mathfrak{D},\mathfrak{M},\mathfrak{G}_{60}$ are strictly contained in $\mathfrak{W}$.
	Furthermore, for every $\mathfrak{P}\in \{\mathfrak{D}_3,\mathfrak{D},\mathfrak{M},\mathfrak{G}_{60}\}$, $ \mathfrak{W}\setminus \mathfrak{P}$ is infinite. 
\end{fact}
\begin{proof}
	By Proposition~\ref{prop:large-mult-Eliahou-jems}, $\mathfrak{M}\subseteq \mathfrak{E}$. Thus $\mathfrak{W}\setminus\mathfrak{E}\subseteq\mathfrak{W}\setminus\mathfrak{M}$.
	Since, by Fact~\ref{fact:infinite-non-eliahou}, $\mathfrak{W}\setminus \mathfrak{E}$ is infinite, it follows that $ \mathfrak{W}\setminus \mathfrak{M}$ is infinite.
    The reader will have no difficulties in giving examples showing that also $\mathfrak{W}\setminus \mathfrak{D}_3$, $\mathfrak{W}\setminus \mathfrak{D}$ and $\mathfrak{W}\setminus \mathfrak{G}_{60}$ are infinite.
\end{proof}

In what follows I will define a partial order on properties (about numerical semigroups). It can be used to partially order classes of numerical semigroups, or even results taking into account the above correspondences. Note that I do not want to make any judgement on the results and even less on their proofs. It may well happen that the ideas involved in the proof of a given result will in the future have a greater impact than the ideas involved in a proof of one of its generalizations.

A property $\mathcal{P}$ is said to be a \emph{generalization} of a property $\mathcal{Q}$ if all the semigroups satisfying  $\mathcal{Q}$  also satisfy $\mathcal{P}$, that is, $\mathfrak{Q}\subseteq \mathfrak{P}$ (or $\mathfrak{Q}\setminus \mathfrak{P}$ is empty).
It is clear that a result that proves a generalization is better, but this can not be said in a definitive way when the arguments in the proofs are different.
Except in some obvious cases (such as happens in several subsections of Section~\ref{sec:particular-cases}), just comparing through set inclusion is not of great help.

A property $\mathcal{P}$ is a \emph{quasi generalization} of a property $\mathcal{Q}$ if all but finitely many numerical semigroups satisfying $\mathcal{Q}$  also satisfy $\mathcal{P}$, that is, $\mathfrak{Q}\setminus \mathfrak{P}$ is either empty or finite. 
It is straightforward to check that quasi-generalization is a partial order in the set of properties on numerical semigroups.
The notation $\mathcal{Q}\prec\mathcal{P}$ is used for ``$\mathcal{P}$ is a quasi generalization of~$\mathcal{Q}$''.
In symbols:\quad $\mathcal{Q}\prec\mathcal{P} \text{ if and only if } \left|\mathfrak{Q}\setminus \mathfrak{P}\right| < \infty$.

I am far from saying that properties that are quasi generalized by others are not important (even without taking the proofs into account). According to this definition, any property defining an infinite class of numerical semigroups quasi generalizes all the properties defining finite classes. For instance, the property $\mathcal{G}_{60}$ is quasi generalized by the properties associated to the results stated in previous section that define infinite classes. But none of these results generalizes $\mathcal{G}_{60}$ (as the reader can easily check), which, from my point of view, makes it a property of high interest.

I encourage anyone who finds a new property (such that all numerical semigroups in the class of semigroups satisfying that property are Wilf) to compare it with other properties for quasi generalization. Observing that it is not known any quasi generalization $\mathcal{P}$ of the property under consideration such that $\mathfrak{P}\subseteq\mathfrak{W}$ probably will count in favour of the results obtained.

 My aim now is to compare, for quasi generalization, the properties for which a name was given: $\mathcal{S},\mathcal{W}, \mathcal{E},\mathcal{D}_3,\mathcal{D},\mathcal{M},$ and $\mathcal{G}_{60}$. 
They do not form a chain, as it follows from next result. The impatient reader may already take a look at the lattice depicted in Figure~\ref{fig:lattice}.

\begin{proposition}\label{prop:M-D-non-comparable}
	$\mathcal{M}$ and $\mathcal{D}$ are not comparable under quasi generalization.
\end{proposition}
\begin{proof}
	For a given $m>1$, $S= \langle m\rangle_{mk}=
	\langle m, km+1,\ldots,km+m-1\rangle$ is a semigroup of maximal embedding dimension, thus satisfies $\mathcal{D}$ and, for $k > 3$, $S$ does not satisfy $\mathcal{M}$. As there are infinitely many such semigroups, it follows that $\mathfrak{D}\setminus\mathfrak{M}$ is infinite and therefore $\mathcal{D}\nprec\mathcal{M}$.
	
	It remains to prove that $\mathcal{M}\nprec\mathcal{D}$, which amounts to show that there are infinitely many semigroups in $\mathfrak{M}$ with small embedding dimension.
	
	The proof of this fact begins with a trivial observation. As usual, for sets of integers $A$ and $B$, $A+B$ denotes the set $\{a+b\mid a\in A, b\in B\}$. 
	\begin{claim}\label{claim:X}
		Let $X =\{0,1,2,3\}\cup \{7k\mid k\in \mathbb{N}\}$. Then $\mathbb{N}\subseteq X+X+X$.
	\end{claim}
	\textbf{Proof of the claim.}
		Clearly $\{0,\ldots,6\}\subseteq X+X$.
		By the Euclidean algorithm every integer can be written in the form $7k+\rho$, with $k$ an integer and $\rho\in \{0,\ldots,6\}$. Consequently, any nonnegative integer belongs to $X+X+X$, which proves the claim.
	\medskip

	Let $m$ be a positive integer and let $Y=\{m, m+1,m+2,m+3\}\cup \{7k+m\mid 0\le k\le \lfloor \frac{m}{7}\rfloor\}$. Consider the semigroup $S=\langle Y\rangle$. Example~\ref{example:m28-and-m80} helps to visualize it for two possible values of~$m$.
	
	The set $Y+Y+Y$ contains $m$ consecutive integers. Since $Y+Y+Y=\{3m\}+X+X+X$, the argument used in in the proof of Claim~\ref{claim:X} allows to conclude that $Y+Y+Y$ contains $\big\{3m,\ldots,3m+\lfloor \frac{m}{7}\rfloor+3\big\}$. It follows that $\conductoroper(S)\le 3\multiplicityoper(S)$. Thus $S$ satisfies $\mathcal{M}$.
	
	It is straightforward to check that the embedding dimension of $S$ is $\lfloor\frac{m}{7}\rfloor+5$.
	As $\frac{m}{7}+5<\frac{m}{3}$ if and only if $7m-3m > 3\times 7\times 5$. This implies $m > 26.25$. Thus one concludes that, for $m>27$, $S$ does not satisfy $\mathcal{D}$.
\end{proof}	
\begin{example}\label{example:m28-and-m80}
	Let $Y=\{m, m+1,m+2,m+3\}\cup \{7k+m\mid 0\le k\le \lfloor \frac{m}{7}\rfloor\}$ be the set introduced in the proof of Proposition~\ref{prop:M-D-non-comparable}. Consider the semigroup $S=\langle Y\rangle$. Figures~\ref{fig:M-not-ED-m28} and~\ref{fig:M-not-ED-m80} give a pictorial representation for the cases $m=28$ and $m=80$ (in the latter case only the shape is drawn).

	The following \textsf{GAP} code can be used to give the semigroups.
	\begin{verbatim}
	m := 28;;
	small_gens := [m,m+1,m+2,m+3];
	other_gens := List([1..Int(m/7)], k -> 7*k+m);
	ns := NumericalSemigroup(Union(small_gens,other_gens));;
	\end{verbatim}	
	Then an adaptation of the GAP-code~\ref{code:draw} can be used to get the \textsf{TikZ} code. In order to obtain an image just showing the shape, one can use the options in GAP-code~\ref{code:shape}.
	
	\begin{figure}[h]
	\begin{center}
		\includegraphics[width=\textwidth]{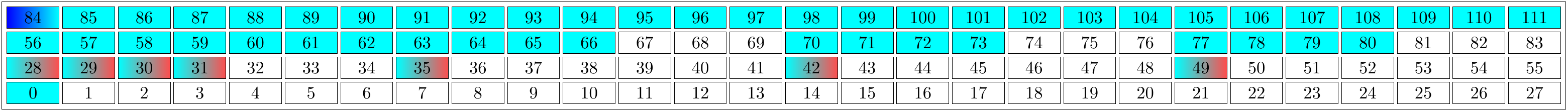}
	\end{center}
	\caption{Pictorial representation of the semigroup obtained with $m=28$. \label{fig:M-not-ED-m28}}
\end{figure} 
	\begin{figure}[h]
		\begin{center}
			\includegraphics[width=\textwidth]{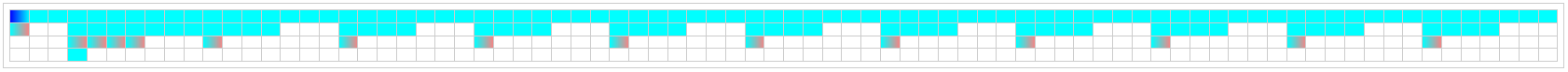}
		\end{center}
		\caption{Shape of the numerical semigroup obtained with $m=80$. \label{fig:M-not-ED-m80}}
	\end{figure} 	
\end{example}

Most of the indicated relations in the lattice represented in Figure~\ref{fig:lattice} have been treated along the text in this section. That  $\mathcal{D}$ and $\mathcal{M}$ are not comparable under quasi-generalization is shown in Proposition~\ref{prop:M-D-non-comparable}. Thus, the following has been proved:

\begin{proposition}\label{prop:not-chain}
	With the notation introduced, one has the lattice in Figure~\ref{fig:lattice}.
\end{proposition}
\thisfloatsetup{floatwidth=.35\hsize,capbesidewidth=sidefil,capposition=beside,capbesideposition=right}
%\thisfloatsetup[figure]{floatwidth=.35\hsize}
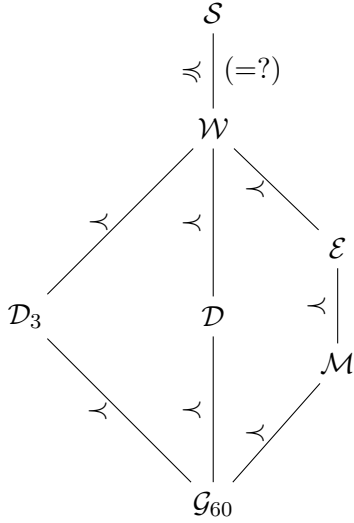
\begin{figure}[h]
	\centering
	\caption{Latice of some numerical semigroup properties (for quasi-generalization)}\label{fig:lattice}
	\begin{tikzpicture}[node distance=2.5cm]
	\node(S)        {$\mathcal{S}$};
	\node(W)  [below = 1cm of S]      {$\mathcal{W}$};
	\node(E)  [below right=1.5cm of W]     {$\mathcal{E}$};
	\node(M)  [below = 1cm of E]     {$\mathcal{M}$};
	\node(D)  [below of = W]        {$\mathcal{D}$};
	\node(D3)  [left of = D]        {$\mathcal{D}_3$};
	\node(G)   [below of = D]         {$\mathcal{G}_{60}$};
	
	\draw(S) -- node[left]{$\preccurlyeq$} node [right]{($=?$)}(W);
	\draw(W) -- node[left]{$\prec$}(D3);
	\draw(W) -- node[left]{$\prec$}(D);
	\draw(W) -- node[left]{$\prec$}(E);
	\draw(D3) -- node[left]{$\prec$}(G);
	\draw(D) -- node[left]{$\prec$}(G);
	\draw(E) -- node[left]{$\prec$}(M);
	\draw(M) -- node[left]{$\prec$}(G);
	\end{tikzpicture}
\end{figure}

Other comparisons could be made. As an example, fix a Wilf semigroup $S$ (possibly with small multiplicity when compared to the conductor). 
It is easy to check that $\multiplicityoper(D(S,a))=\multiplicityoper(S)+a$ and that $\conductoroper(D(S,a))=\conductoroper(S)+a$. Thus, $\frac{\multiplicityoper(D(S,a))}{\conductoroper(D(S,a))}= \frac{\multiplicityoper(S)+a}{\conductoroper(S)+a}$ tends to $1$ when $a$ tends to infinity. In particular, from a certain point on, the quotient $\frac{\multiplicityoper}{\conductoroper}$ is greater than $1/3$ and so, from that point on, all the semigroups satisfy $\mathcal{M}$. Therefore, $\mathcal{M}$ quasi generalizes the property corresponding to Proposition~\ref{prop:dilations}, for any fixed $S$.
\medskip
 
Denote by $\mathcal{P}_4$ the property associated to Proposition~\ref{prop:very-large-mult} with $\rho=4$.
As there are infinitely many semigroups satisfying $\mathcal{D}$ whose multiplicity is even, we get that $\mathcal{D}\nprec\mathcal{P}_4$. On the other hand, it is straightforward to check that there are infinitely many semigroups satisfying $\mathcal{P}_4$ and with small embedding dimension (less than $\multiplicityoper/3$). Thus $\mathcal{D}\nsucc\mathcal{P}_4$, and we conclude that $\mathcal{D}$ and $\mathcal{P}_4$, are not comparable under quasi generalization.
\medskip

Let $\embeddingdimensionoper\ge 10$ be an integer and denote by $\mathfrak{S}_{(\embeddingdimensionoper,\conductoroper)}$ the class of numerical semigroups with embedding dimension $\embeddingdimensionoper$ and conductor $\conductoroper$. Denote by $\ratiooper=\ratiooper(S)$ the ratio (second smallest primitive) of a numerical semigroup $S$. For fixed $\embeddingdimensionoper$ and $\conductoroper$, consider the set 
\[\mathfrak{R}_{(\embeddingdimensionoper,\conductoroper)}=\left\{S\in \mathfrak{S}_{(\embeddingdimensionoper,\conductoroper)}\mid \multiplicityoper \le \frac{8\embeddingdimensionoper^2}{25}+\frac{\embeddingdimensionoper}{5}-\frac{5}{4}\text{ and } \ratiooper >\left\lfloor\frac{\conductoroper+\multiplicityoper}{3}\right\rfloor\right\}.\]

Since no primitive of a numerical semigroup exceeds $\conductoroper+\multiplicityoper-1$, the ratio of a semigroup of embedding dimension $\embeddingdimensionoper$ must be at most  $\conductoroper+\multiplicityoper-\embeddingdimensionoper+1$.

Note that the class of numerical semigroups satisfying Equation~\ref{eq:prop-ratio} in Proposition~\ref{prop:ratio} is: \[\mathfrak{R}_{\embeddingdimensionoper}=\bigcup_{\conductoroper\ge\multiplicityoper}\mathfrak{R}_{(\embeddingdimensionoper,\conductoroper)}.\]

Consider now the class $\mathfrak{R}_{\embeddingdimensionoper}\setminus(\mathfrak{M}\cup \mathfrak{D})$. In set notation it may be written as follows:

\[\left\{S\in \mathfrak{S}\mid 3\embeddingdimensionoper < \multiplicityoper \le \frac{8}{25}\embeddingdimensionoper^2+\frac{1}{5}\embeddingdimensionoper-\frac{5}{4}, \conductoroper > 3\multiplicityoper \text{ and } \left\lfloor\frac{\conductoroper+\multiplicityoper}{3}\right\rfloor<\ratiooper \le \conductoroper+\multiplicityoper-\embeddingdimensionoper+1\right\}.\]
A natural question is whether this class is finite, that is, the disjunction of the properties $\mathcal{M}$ and $\mathcal{D}$ quasi generalizes the property associated to Proposition~\ref{prop:ratio} (with $\embeddingdimensionoper$ fixed).
Apparently there is no bound for the conductor, but a big conductor will force a big ratio. On the other hand, a small embedding dimension and a big ratio leads to a huge conductor. Is the starting ``conductor'' big enough?

% ===========================================

%%%%%%%%%%%%%%%%%%%%%%%%%%%%%%%%%%%% 
%\section{drafts}
%\input{sections/olds/drafts} %includes references
%\input{sections/olds/visualization}
%\input{sections/olds/counting}
%%%%%%%%%%%%%%%%%%%%%%%%%%%%%%%%%%%% 
`%%%%%%%%%%% Acknowledgements %%%%%%%%
\subsubsection*{Acknowledgements}\label{sec:acknowledgements}
%=====================
I would like to thank my colleges ate the FCUP's Mathematics department who made possible for me to benefit of a sabbatical year. The hospitality found in the Instituto de Matemáticas de la Universidad de Granada (IEMath-GR) was amazing, which made the writing of this paper a lot easier. Many thanks to IEMath-GR and especially to Pedro García-Sanchez who made possible this stay in Granada to happen. I want also to thank Claude Marion whose interest shown on the topic highly contributed for my decision on writing this survey. His comments greatly contributed to improve the paper.
%I am also indebted to him (at least as much as the reader) for the numerous corrections of English. 
%=====================

%%%%%%%%%%% Acknowledgements %%%%%%%%
% in order to use bibtex comment the following line and uncomment the others
%% uncomment to use biblatex  
%\printbibliography
%%
%%uncomment to use bibtex
%\bibliographystyle{plainurl}
%\bibliography{../../../../bib/NumericalS-bib/numericals.bib,../../../../bib/NumericalS-bib/preprints.bib,../../../../bib/NumericalS-bib/software.bib,../../../../bib/NumericalS-bib/slides.bib}

\begin{thebibliography}{10}

\bibitem{Backelin1990MS-number}
J\"orgen Backelin.
\newblock On the number of semigroups of natural numbers.
\newblock {\em Math. Scand.}, 66(2):197--215, 1990.
%\newblock URL: \url{https://doi.org/10.7146/math.scand.a-12304},
\newblock \href  {http://dx.doi.org/10.7146/math.scand.a-12304}
  {\path{doi:10.7146/math.scand.a-12304}}.

\bibitem{BarucciStrazzanti2018SF-Dilatations}
V.~Barucci and F.~Strazzanti.
\newblock Dilatations of numerical semigroups.
\newblock {\em Semigroup Forum}, Feb 2018.
%\newblock URL: \url{https://doi.org/10.1007/s00233-018-9922-9}, 
\newblock \href {http://dx.doi.org/10.1007/s00233-018-9922-9}
  {\path{doi:10.1007/s00233-018-9922-9}}.

\bibitem{Barucci2009}
Valentina Barucci.
\newblock On propinquity of numerical semigroups and one-dimensional local
  {C}ohen {M}acaulay rings.
\newblock In {\em Commutative algebra and its applications}, pages 49--60.
  Walter de Gruyter, Berlin, 2009.

\bibitem{BarucciFroeberg1997JA-One-dimensional}
Valentina Barucci and Ralf Fr\"{o}berg.
\newblock One-dimensional almost {G}orenstein rings.
\newblock {\em J. Algebra}, 188(2):418--442, 1997.
%\newblock URL: \url{https://doi.org/10.1006/jabr.1996.6837}, 
\newblock \href {http://dx.doi.org/10.1006/jabr.1996.6837}
  {\path{doi:10.1006/jabr.1996.6837}}.

\bibitem{Bras-Amoros2008SF-Fibonacci}
Maria Bras-Amor\'{o}s.
\newblock Fibonacci-like behavior of the number of numerical semigroups of a
  given genus.
\newblock {\em Semigroup Forum}, 76(2):379--384, 2008.
%\newblock URL: \url{https://doi.org/10.1007/s00233-007-9014-8}, 
\newblock \href {http://dx.doi.org/10.1007/s00233-007-9014-8}
  {\path{doi:10.1007/s00233-007-9014-8}}.

\bibitem{Brauer1942AJM-problem}
Alfred Brauer.
\newblock On a problem of partitions.
\newblock {\em Amer. J. Math.}, 64:299--312, 1942.
%\newblock URL: \url{https://doi.org/10.2307/2371684}, 
\newblock \href {http://dx.doi.org/10.2307/2371684} {\path{doi:10.2307/2371684}}.

\bibitem{BrunsGarcia-SanchezONeill2019}
Winfried Bruns, Pedro~A. García-Sánchez, and Christopher O'Neill, 2019.
\newblock Personal communication by Pedro A. García-Sánchez.

\bibitem{Curtis1990MS-formulas}
Frank Curtis.
\newblock On formulas for the {F}robenius number of a numerical semigroup.
\newblock {\em Math. Scand.}, 67(2):190--192, 1990.
%\newblock URL: \url{https://doi.org/10.7146/math.scand.a-12330}, 
\newblock \href {http://dx.doi.org/10.7146/math.scand.a-12330}
  {\path{doi:10.7146/math.scand.a-12330}}.

\bibitem{Numericalsgps-2018}
M.~Delgado, P.A. García-Sánchez, and J.~Morais.
\newblock {\em Numericalsgps -- a {GAP} package on numerical semigroups}, 3
  2018.
\newblock Version number 1.1.7.
\newblock URL: \url{http://www.gap-system.org/Packages/numericalsgps.html}.

\bibitem{IntPic-2017}
Manuel Delgado.
\newblock {\em IntPic -- a GAP package for drawing integers}, 9 2017.
\newblock Version 0.2.3.
\newblock URL: \url{http://www.gap-system.org/Packages/intpic.html}.

\bibitem{Delgado2018MZ-question}
Manuel Delgado.
\newblock On a question of {E}liahou and a conjecture of {W}ilf.
\newblock {\em Math. Z.}, 288(1-2):595--627, 2018.
%\newblock URL: \url{https://doi.org/10.1007/s00209-017-1902-3}, 
\newblock \href {http://dx.doi.org/10.1007/s00209-017-1902-3}
  {\path{doi:10.1007/s00209-017-1902-3}}.

\bibitem{Dhayni2017phd-Problems}
Mariam Dhayni.
\newblock {\em Problems in numerical semigroups}.
\newblock PhD thesis, Université d’Angers, 2017.
\newblock URL: \url{https://www.theses.fr/2017ANGE0041.pdf}.

\bibitem{Dhayni2018PJM-Wilfs}
Mariam Dhayni.
\newblock Wilf’s conjecture for numerical semigroups.
\newblock {\em Palest. J. Math.}, 7(2):385--396, 2018.

\bibitem{DobbsMatthews2006}
David~E. Dobbs and Gretchen~L. Matthews.
\newblock On a question of {W}ilf concerning numerical semigroups.
\newblock In {\em Focus on commutative rings research}, pages 193--202. Nova
  Sci. Publ., New York, 2006.

\bibitem{Eliahou2017-Umea}
Shalom Eliahou.
\newblock A graph-theoretic approach to wilf’s conjecture.
\newblock Slides presented at the Meeting of the Catalan, Spanish, Swedish Math
  Societies, 6 2017.
\newblock URL: \url{http://www.ugr.es/~semigrupos/Umea-2017/}.

\bibitem{Eliahou2018JEMS-Wilfs}
Shalom Eliahou.
\newblock Wilf's conjecture and {M}acaulay's theorem.
\newblock {\em J. Eur. Math. Soc. (JEMS)}, 20(9):2105--2129, 2018.
%\newblock URL: \url{https://doi.org/10.4171/JEMS/807}, 
\newblock \href {http://dx.doi.org/10.4171/JEMS/807} {\path{doi:10.4171/JEMS/807}}.

\bibitem{EliahouFromentin2018SF-misses}
Shalom Eliahou and Jean Fromentin.
\newblock Near-misses in wilf's conjecture.
\newblock {\em Semigroup Forum}, Feb 2018.
%\newblock URL: \url{https://doi.org/10.1007/s00233-018-9926-5}, 
\newblock \href {http://dx.doi.org/10.1007/s00233-018-9926-5}
  {\path{doi:10.1007/s00233-018-9926-5}}.

\bibitem{EliahouMarin-Aragon2019}
Shalom Eliahou and Daniel Marín-Aragón, 2019.
\newblock Personal communication by Shalom Eliahou.

\bibitem{FroebergGottliebHaeggkvist1987SF-numerical}
R.~Fr\"oberg, C.~Gottlieb, and R.~H\"aggkvist.
\newblock On numerical semigroups.
\newblock {\em Semigroup Forum}, 35(1):63--83, 1987.
%\newblock URL: \url{https://doi.org/10.1007/BF02573091}, 
\newblock \href {http://dx.doi.org/10.1007/BF02573091} {\path{doi:10.1007/BF02573091}}.

\bibitem{FromentinHivert2016MC-Exploring}
Jean Fromentin and Florent Hivert.
\newblock Exploring the tree of numerical semigroups.
\newblock {\em Math. Comp.}, 85(301):2553--2568, 2016.
%\newblock URL: \url{http://dx.doi.org/10.1090/mcom/3075}, 
\newblock \href {http://dx.doi.org/10.1090/mcom/3075} {\path{doi:10.1090/mcom/3075}}.

\bibitem{GAP4-2018}
The GAP~Group.
\newblock {\em {GAP -- Groups, Algorithms, and Programming, Version 4.9.1}}, 5
  2018.
\newblock URL: \url{https://www.gap-system.org}.

\bibitem{GeroldingerHalter-Koch2006}
Alfred Geroldinger and Franz Halter-Koch.
\newblock Non-unique factorizations: a survey.
\newblock In {\em Multiplicative ideal theory in commutative algebra}, pages
  207--226. Springer, New York, 2006.
%\newblock URL: \url{https://doi.org/10.1007/978-0-387-36717-0_13}, 
\newblock \href {http://dx.doi.org/10.1007/978-0-387-36717-0_13}
  {\path{doi:10.1007/978-0-387-36717-0_13}}.

\bibitem{Kaplan2012JPAA-Counting}
Nathan Kaplan.
\newblock Counting numerical semigroups by genus and some cases of a question
  of {W}ilf.
\newblock {\em J. Pure Appl. Algebra}, 216(5):1016--1032, 2012.
%\newblock URL: \url{http://dx.doi.org/10.1016/j.jpaa.2011.10.038}, 
\newblock \href {http://dx.doi.org/10.1016/j.jpaa.2011.10.038}
  {\path{doi:10.1016/j.jpaa.2011.10.038}}.

\bibitem{Kaplan2017AMM-Counting}
Nathan Kaplan.
\newblock Counting numerical semigroups.
\newblock {\em Amer. Math. Monthly}, 124(9):862--875, 2017.
%\newblock URL: \url{https://doi.org/10.4169/amer.math.monthly.124.9.862}, 
\newblock \href {http://dx.doi.org/10.4169/amer.math.monthly.124.9.862}
  {\path{doi:10.4169/amer.math.monthly.124.9.862}}.

\bibitem{Kunz2016SF-type}
E.~Kunz.
\newblock On the type of certain numerical semigroups and a question of {W}ilf.
\newblock {\em Semigroup Forum}, 93(1):205--210, 2016.
%\newblock URL: \url{https://doi.org/10.1007/s00233-015-9755-8}, 
\newblock \href {http://dx.doi.org/10.1007/s00233-015-9755-8}
  {\path{doi:10.1007/s00233-015-9755-8}}.

\bibitem{KunzWaldi2017JA-deviation}
E.~Kunz and R.~Waldi.
\newblock On the deviation and the type of certain local {C}ohen-{M}acaulay
  rings and numerical semigroups.
\newblock {\em J. Algebra}, 478:397--409, 2017.
%\newblock URL: \url{https://doi.org/10.1016/j.jalgebra.2017.01.041}, 
\newblock \href {http://dx.doi.org/10.1016/j.jalgebra.2017.01.041}
  {\path{doi:10.1016/j.jalgebra.2017.01.041}}.

\bibitem{Matthews2005I-integers}
Gretchen~L. Matthews.
\newblock On integers nonrepresentable by a generalized arithmetic progression.
\newblock {\em Integers}, 5(2):A12, 6, 2005.

\bibitem{MoscarielloSammartano2015MZ-conjecture}
Alessio Moscariello and Alessio Sammartano.
\newblock On a conjecture by {W}ilf about the {F}robenius number.
\newblock {\em Math. Z.}, 280(1-2):47--53, 2015.
%\newblock URL: \url{https://doi.org/10.1007/s00209-015-1412-0}, 
\newblock \href {http://dx.doi.org/10.1007/s00209-015-1412-0}
  {\path{doi:10.1007/s00209-015-1412-0}}.

\bibitem{Ramirez-Alfonsin2005Book-Diophantine}
J.~L. Ramírez-Alfonsín.
\newblock {\em The {D}iophantine {F}robenius problem}, volume~30 of {\em Oxford
  Lecture Series in Mathematics and its Applications}.
\newblock Oxford University Press, Oxford, 2005.
%\newblock URL: \url{https://doi.org/10.1093/acprof:oso/9780198568209.001.0001},
\newblock \href {http://dx.doi.org/10.1093/acprof:oso/9780198568209.001.0001}
  {\path{doi:10.1093/acprof:oso/9780198568209.001.0001}}.

\bibitem{RosalesGarcia2009Book-Numerical}
J.~C. Rosales and P.~A. García~Sánchez.
\newblock {\em Numerical semigroups}, volume~20 of {\em Developments in
  Mathematics}.
\newblock Springer, New York, 2009.
%\newblock URL: \url{http://dx.doi.org/10.1007/978-1-4419-0160-6}, 
\newblock \href {http://dx.doi.org/10.1007/978-1-4419-0160-6}
  {\path{doi:10.1007/978-1-4419-0160-6}}.

\bibitem{Sammartano2012SF-Numerical}
Alessio Sammartano.
\newblock Numerical semigroups with large embedding dimension satisfy {W}ilf's
  conjecture.
\newblock {\em Semigroup Forum}, 85(3):439--447, 2012.
%\newblock URL: \url{https://doi.org/10.1007/s00233-011-9370-2}, 
\newblock \href {http://dx.doi.org/10.1007/s00233-011-9370-2}
  {\path{doi:10.1007/s00233-011-9370-2}}.

\bibitem{Spirito2017ae-Wilfs}
Dario {Spirito}.
\newblock {Wilf's conjecture for numerical semigroups with large second
  generator}.
\newblock {\em arXiv e-prints}, page arXiv:1710.09245, October 2017.
\newblock \href {http://arxiv.org/abs/1710.09245} {\path{arXiv:1710.09245}}.

\bibitem{Sylvester1984ET-Mathematical}
J.~J. Sylvester.
\newblock Mathematical questions with their solutions.
\newblock {\em Educational Times}, 41:21, 1884.
\newblock Solution by W.J. Curran Sharp.

\bibitem{TaoVu2006Book-Additive}
Terence Tao and Van Vu.
\newblock {\em Additive combinatorics}, volume 105 of {\em Cambridge Studies in
  Advanced Mathematics}.
\newblock Cambridge University Press, Cambridge, 2006.
%\newblock URL: \url{https://doi.org/10.1017/CBO9780511755149}, 
\newblock \href {http://dx.doi.org/10.1017/CBO9780511755149}
  {\path{doi:10.1017/CBO9780511755149}}.

\bibitem{Wilf1978AMM-circle}
Herbert~S. Wilf.
\newblock A circle-of-lights algorithm for the ``money-changing problem''.
\newblock {\em Amer. Math. Monthly}, 85(7):562--565, 1978.
%\newblock URL: \url{http://dx.doi.org/10.2307/2320864}, 
\newblock \href {http://dx.doi.org/10.2307/2320864} {\path{doi:10.2307/2320864}}.

\bibitem{Zhai2012SF-Fibonacci}
Alex Zhai.
\newblock Fibonacci-like growth of numerical semigroups of a given genus.
\newblock {\em Semigroup Forum}, 86(3):634--662, 2013.
%\newblock URL: \url{http://dx.doi.org/10.1007/s00233-012-9456-5}, 
\newblock \href {http://dx.doi.org/10.1007/s00233-012-9456-5}
  {\path{doi:10.1007/s00233-012-9456-5}}.

\bibitem{Zhao2010SF-Constructing}
Yufei Zhao.
\newblock Constructing numerical semigroups of a given genus.
\newblock {\em Semigroup Forum}, 80:242--254, 2010.
%\newblock URL: \url{http://dx.doi.org/10.1007/s00233-009-9190-9}, 
\newblock \href {http://dx.doi.org/10.1007/s00233-009-9190-9}
  {\path{doi:10.1007/s00233-009-9190-9}}.

\end{thebibliography}
%%%%%%%%%

\end{document}